\documentstyle[11pt,amssymb,amsmath,amscd,amsbsy,amsfonts,amsthm]{article}

\input xy
\xyoption{all}

\newtheorem{thm}{Theorem}[section]
\newtheorem{lem}[thm]{Lemma}
\newtheorem{prop}[thm]{Proposition}
\newtheorem{cor}[thm]{Corollary}
\newtheorem{rem}[thm]{Remark}
\newtheorem{dfn}[thm]{Definition}

\DeclareMathOperator{\Z}{\mathbb{Z}}
\DeclareMathOperator{\Gm}{{\mathbb G}_m}

\DeclareMathOperator{\Ima}{im}

\title{\textsc{Motivic cohomology of the Nisnevich classifying space of even Clifford groups}}
\author{Fabio Tanania}
\date{}

\begin{document}
	
	\maketitle
	
		\begin{abstract}
	In this paper, we consider the split even Clifford group $\Gamma^+_n$ and compute the mod 2 motivic cohomology ring of its Nisnevich classifying space. The description we obtain is quite similar to the one provided for spin groups in \cite{T2}. The fundamental difference resides in the behaviour of the second subtle Stiefel-Whitney class that is non-trivial for even Clifford groups, while it vanished in the spin-case.
	\end{abstract}

	\section{Introduction}
	
	Subtle characteristic classes were introduced by Smirnov and Vishik in \cite{SV} to approach the classification of quadratic forms by using motivic homotopical techniques. In particular, these characteristic classes arise as elements of the motivic cohomology ring of the Nisnevich classifying space $BG$ of a linear algebraic group $G$ over a field $k$. They naturally provide invariants for Nisnevich locally trivial $G$-torsors, which take value in the motivic cohomology of the base. What is probably more interesting is that they also provide invariants for \'etale locally trivial $G$-torsors, which take value this time in a more complicated and informative object, namely the motivic cohomology of the \v{C}ech simplicial scheme of the torsor under study. 
	
	In \cite{SV}, the authors compute the motivic cohomology ring with $\Z/2$-coefficients of $BO_n$, i.e. the Nisnevich classifying space of the split orthogonal group. Similarly to the topological picture, this cohomology ring is a polynomial algebra over the motivic cohomology of the ground field generated by certain classes $u_1,\dots,u_n$ called subtle Stiefel-Whitney classes. These invariants detect the power $I^n$ of the fundamental ideal of the Witt ring a quadratic form belongs to. In particular, the triviality of all subtle Stiefel-Whitney classes implies the triviality of the quadratic form itself. Besides, from the computation of $H(BO_n)$ it follows that the mod 2 motivic cohomology of $BSO_n$ is also a polynomial algebra generated by all subtle Stiefel-Whitney classes but the first.
	
	Following \cite{SV}, we studied the motivic cohomology rings of the Nisnevich classifying spaces of unitary groups in \cite{T1}, of spin groups in \cite{T2} and of projective general linear groups in \cite{T3}. This paper is a natural follow-up of \cite{T2}. In fact, we focus here on computing the motivic cohomology of the Nisnevich classifying space of even Clifford groups. These algebraic groups are closely related to spin groups. On the level of torsors, this is visible from the fact that a spin-torsor yields a quadratic form in $I^3$ through a surjective map with trivial kernel, while the torsors of the even Clifford groups are exactly the quadratic forms in $I^3$. 
	
	The topological counterpart of the even Clifford group $\Gamma^+_n$ is the Lie group $Spin^c(n)$. The singular cohomology of the classifying space of $Spin^c(n)$ was computed by Harada and Kono in \cite{HK}. The main result we obtain in this article is a motivic version of \cite[Theorem 3.5]{HK}. More precisely, we prove the following.
		\begin{thm}\label{intro}
		Let $k$ be a field of characteristic different from $2$ containing a square root of $-1$. Then, for any $n \geq 2$, there exists a cohomology class $e_{2^{l(n)}}$ in bidegree $(2^{l(n)-1})[2^{l(n)}]$ such that the natural homomorphism of $H$-algebras 
		$$H(BSO_n)/I^{\circ}_{l(n)} \otimes_H H[e_{2^{l(n)}}] \rightarrow H(B\Gamma^+_n)$$
		is an isomorphism, where $I^{\circ}_{l(n)}$ is the ideal generated by $\theta_1, \dots ,\theta_{l(n)-1}$ and $l(n)=[\frac{n+1}{2}]$.
	\end{thm}
The assumption on the characteristic of the ground field is necessary since the mod $2$ motivic cohomology of the point and the mod $2$ motivic Steenrod algebra are well understood in characteristic different from $2$ (see \cite{V3}). Moreover, we require also that $k$ contains a square root of $-1$, since in this case the action of the mod $2$ motivic Steenrod algebra on the mod $2$ motivic cohomology of the point is trivial, making our computation easier. Anyways, we suspect that a result similar to Theorem \ref{intro} would still hold after dropping this last assumption, but with more complicated relations involving Steenrod operations and $\rho=Sq^1 \tau$, where $\rho$ is the class of $-1$ in mod $2$ Milnor K-theory and $\tau$ is the generator of $H^{0,1} \cong \Z/2$.

The similarity between Theorem \ref{intro} and the computation for the spin-case (see Theorem \ref{bso}) is clear. Nonetheless, a crucial difference is that, while $u_2$ is trivial in $H(BSpin_n)$, it is not in $H(B\Gamma^+_n)$ where the ideal of relations $I^{\circ}_{l(n)}$ is generated by the action of the motivic Steenrod algebra over $u_3=Sq^1u_2$. This also explains the gap between $k(n)$ in Theorem \ref{bso} and $l(n)$ in Theorem \ref{intro}, which is due to discrepancies in the maximal length of the regular sequences in $H(BSO_n)$ obtained by applying certain Steenrod operations to $u_2$ and $u_3$ respectively.

We conclude by pointing out that understanding the motivic cohomology of Nisnevich classifying spaces also helps in obtaining information about the structure of the Chow ring of \'etale classifying spaces $B_{\acute et}G$ (see \cite{To}), which is an interesting object of study that is particularly challenging to fully grasp. For example, the Chow ring of $B_{\acute et}\Gamma^+_n$ has been recently investigated by Karpenko in \cite{K} where he proves a conjecture that allows him, as a consequence, to compute the exponent indexes of spin grassmannians. In our case, we will show how to apply Theorem \ref{intro} to compute the subring generated by Chern classes of the Chow ring mod 2 of $B_{\acute et}\Gamma^+_n$, modulo nilpotents, sheding new light on its complicated structure. \\

    \textbf{Outline.} In Section 2, we report notations and preliminary results that we will use in this paper. In particular, we recall the Thom isomorphism in the triangulated category of motives over a simplicial base, which provides Gysin long exact sequences in motivic cohomology. Besides, we recall definitions and properties of classifying spaces in motivic homotopy theory, as well as the computation of the mod 2 motivic cohomology of $BO_n$, $BSO_n$ and $BSpin_n$. In Section 3, we investigate regular sequences in $H(BSO_n)$ constructed starting from $u_3$ by acting with specific Steenrod operations. Finally, in Section 4, we exploit the Gysin sequence relating $B\Gamma^+_n$ and $BSpin_n$, and the regular sequences studied before, in order to fully compute the mod 2 motivic cohomology of $B\Gamma^+_n$. After that, we also obtain a complete description of the reduced Chern subring of the Chow ring mod 2 of $B_{\acute et}\Gamma^+_n$.\\
	
	\textbf{Acknowledgements.} I wish to sincerely thank Nikita Karpenko for encouraging me to write this paper. I am also grateful to the referee for very helpful comments.
	
	\section{Notations and preliminaries}
	
	We start by fixing some notations we will regularly use in this paper.
	
	\begin{tabular}{c|c}
		$k$ & field of characteristic different from $2$ containing $\sqrt{-1}$\\
		$R$ & commutative ring with identity\\
		$Y_{\bullet}$ & smooth simplicial scheme over $k$\\
		$Spc_*(Y_{\bullet})$  & category of pointed motivic spaces over $Y_{\bullet}$\\
		${\mathcal H}_s(k)$ & simplicial homotopy category over $k$\\
		${\mathcal {DM}}^{-}_{eff}(k,R)$ & triangulated category of effective motives over $k$ with $R$-coefficients\\
		${\mathcal DM}^-_{eff}(Y_{\bullet},R)$ & triangulated category of effective motives over $Y_{\bullet}$ with $R$-coefficients\\
		$T$ & unit object in ${\mathcal {DM}}^{-}_{eff}(k,R)$\\
		$H_{top}(-)$ & singular cohomology with $\Z/2$-coefficients\\
		$H(-)$ & motivic cohomology with $\Z/2$-coefficients\\
		$H$ & motivic cohomology with $\Z/2$-coefficients of $Spec(k)$\\
		$K^M(k)/2$ & Milnor K-theory of $k$ mod 2\\
		$w_i$ & $i$th Stiefel-Whitney class in $H_{top}(BSO_n)$\\
		$u_i$ & $i$th subtle Stiefel-Whitney class in $H(BSO_n)$\\
		$\rho_j$ & the element $Sq^{2^{j-1}}Sq^{2^{j-2}} \dots Sq^2Sq^1w_2$ in $H_{top}(BSO_n)$\\
		$\theta_j$ & the element $Sq^{2^{j-1}}Sq^{2^{j-2}} \dots Sq^2Sq^1u_2$ in $H(BSO_n)$\\
		$\Gamma^+_n$ & even Clifford group
	\end{tabular}\\

The collection of results \cite[Theorem 6.1, Corollary 6.9 and Corollary 7.5]{V3} implies that $H \cong K^M(k)/2[\tau]$, where $\tau$ is the non-trivial class in $H^{0,1} \cong \Z/2$ and $H^{n,n} \cong K_n^M(k)/2$.

Note that, since we are working over a field containing the square root of $-1$, all Steenrod squares $Sq^i$, as defined in \cite{V2}, act trivially on $H$.

Since we will mainly work in the triangulated category of motives over a simplicial scheme defined by Voevodsky in \cite{V1}, we recall a few definitions and propositions about it that will be useful later on to prove our main results.

\begin{dfn}
For any smooth simplicial scheme $Y_{\bullet}$ over $k$, denote by $c:Y_{\bullet} \rightarrow Spec(k)$ the projection to the base. Then, we can define the Tate objects $T(q)[p]$ in ${\mathcal DM}^-_{eff}(Y_{\bullet},R)$ as $c^*(T(q)[p])$. 
\end{dfn}
	
\begin{dfn}
A smooth morphism of smooth simplicial schemes $\pi:X_{\bullet} \rightarrow Y_{\bullet}$ is called coherent if there is a cartesian square
$$
\xymatrix{
	X_j \ar@{->}[r]^{\pi_j} \ar@{->}[d]_{X_{\theta}} & Y_j \ar@{->}[d]^{Y_{\theta}}\\
	X_i \ar@{->}[r]_{\pi_i} & Y_i
}
$$
for any simplicial map $\theta:[i] \rightarrow [j]$. 
\end{dfn}	
	
	Denote by $CC(Y_{\bullet})$ the simplicial set obtained from $Y_{\bullet}$ by applying the functor $CC$ that sends any connected scheme to the point and respects coproducts.

	\begin{prop}\label{Thom1}
		Let $\pi:X_{\bullet} \rightarrow Y_{\bullet}$ be a smooth coherent morphism of smooth simplicial schemes over $k$ and $A$ a smooth $k$-scheme such that:\\
		1) $X_0$ is isomorphic to $Y_0 \times A$ and, under this isomorphism, $\pi_0$ becomes equal to the projection map $Y_0 \times A \rightarrow Y_0$;\\
		2) $H^1(CC(Y_{\bullet}),R^{\times}) \cong 0$;\\
		3) $M(A) \cong T \oplus T(r)[s-1]$ in $ {\mathcal {DM}}_{eff}^-(k,R)$ for arbitrary integers $r$ and $s$.
		
		Then, $M(Cone(\pi)) \cong T(r)[s]$ in $ {\mathcal {DM}}_{eff}^-(Y_{\bullet},R)$ where $Cone(\pi)$ is the cone of $\pi$ in $Spc_*(Y_{\bullet})$. Hence, we get a Thom isomorphism of $H(Y_{\bullet},R)$-modules 
		$$H^{*-s,*'-r}(Y_{\bullet},R) \rightarrow H^{*,*'}(Cone(\pi),R).$$
	\end{prop}
	\begin{proof}
	See \cite[Proposition 4.2]{T2}.
	\end{proof}
	
	\begin{dfn}
	We call Thom class of $\pi$ and denote by $\alpha$ the image of $1$ under the Thom isomorphism of Proposition \ref{Thom1}.
	\end{dfn}

The following result guarantees that the Thom isomorphism from Proposition \ref{Thom1} is functorial.
	
	\begin{prop}\label{Thom2}
	   Suppose there is a cartesian square
		$$
		\xymatrix{
			X_{\bullet} \ar@{->}[r]^{\pi} \ar@{->}[d]_{p_X} & Y_{\bullet} \ar@{->}[d]^{p_Y}\\
			X'_{\bullet} \ar@{->}[r]_{\pi'} & Y'_{\bullet}
		}
		$$
		such that $Y_0$ is connected, $p_X$ and $p_Y$ are smooth, $\pi$ and $\pi'$ are smooth coherent with fiber $A$ satisfying all conditions from Proposition \ref{Thom1}.
		Then, the induced homomorphism in motivic cohomology
		$$H(Cone(\pi'),R) \rightarrow H(Cone(\pi),R)$$
		maps $\alpha'$ to $\alpha$, where $\alpha'$ and $\alpha$ are the respective Thom classes.
	\end{prop}
	\begin{proof}
	See \cite[Proposition 4.3 and Corollary 4.4]{T2}.
	\end{proof}
	
	We now recall from \cite{MV} the definitions of the Nisnevich and \'etale classifying spaces of linear algebraic groups.
	
	Let $G$ be a linear algebraic group over $k$ and $EG$ the simplicial scheme defined by $(EG)_n=G^{n+1}$, with partial projections as face maps and partial diagonals as degeneracy maps. The space $EG$ is endowed with a right free $G$-action provided by the operation in $G$. 
	
	\begin{dfn}
	The Nisnevich classifying space of $G$ is the quotient $BG=EG/G$.
	\end{dfn} 

The morphism of sites $\pi:(Sm/k)_{\acute{e}t} \rightarrow (Sm/k)_{Nis}$ induces an adjunction between simplicial homotopy categories 
	\begin{align*}
		{\mathcal H}_s((S&m/k)_{\acute{e}t})\\
		\pi^* \uparrow & \downarrow R\pi_*\\
		{\mathcal H}_s((S&m/k)_{Nis}).
	\end{align*}

 \begin{dfn}
The \'etale classifying space of $G$ is defined by $B_{\acute{e}t}G=R\pi_*\pi^*BG$.
 \end{dfn}
	
	Let $H$ be an algebraic subgroup of $G$. Then, we can define the simplicial scheme $\widehat{B}H=EG/H$ with respect to the embedding $H \hookrightarrow G$. Denote by $j$ the induced morphism $BH \rightarrow \widehat{B}H$. 
	
 \begin{prop}\label{BG}
		Let $H \hookrightarrow G$ be such that all rationally trivial $H$-torsors and $G$-torsors are Zariski-locally trivial. If the map $Hom_{{\mathcal H}_s(k)}(Spec(K),B_{\acute{e}t}H) \rightarrow Hom_{{\mathcal H}_s(k)}(Spec(K),B_{\acute{e}t}G)$ has trivial kernel for any finitely generated field extension $K$ of $k$, then $j$ is an isomorphism in ${\mathcal H}_s(k)$. 
	\end{prop}
	\begin{proof}
	See \cite[Proposition 5.1, Corollary 5.2 and Proposition 5.3]{T2}.
	\end{proof}

\begin{rem}\label{menumal}
	\normalfont
	Note that the obvious map $\pi: \widehat{B}H \rightarrow BG$ is smooth coherent with fiber $G/H$.
\end{rem}

By using the Gysin sequence induced by the Thom isomorphism, one can compute by induction the following motivic cohomology rings.
	
	\begin{thm}\label{bso}
		The motivic cohomology rings of $BO_n$, $BSO_n$ and $BSpin_n$ are respectively given by
		$$H(BO_n) \cong H[u_1,\dots,u_n],$$
		$$H(BSO_n) \cong H[u_2,\dots,u_n],$$
		$$H(BSpin_n) \cong H(BSO_n)/I_{k(n)} \otimes_H H[v_{2^{k(n)}}],$$
		where the {\emph i}th subtle Stiefel-Whitney class $u_i$ is in bidegree $([\frac{i}{2}])[i]$, the class $v_{2^{k(n)}}$ is in bidegree $(2^{k(n)-1})[2^{k(n)}]$, $I_{k(n)}$ is the ideal generated by $\theta_0, \dots ,\theta_{k(n)-1}$ and $k(n)$ depends on $n$ as in the following table.
		\begin{center}
			\begin{tabular}{|c|c|}
				\hline
				\textbf{n} &\textbf{k(n)}\\
				\hline
				8l+1 &4l\\
				\hline
				8l+2 &4l+1\\
				\hline
				8l+3 &4l+2\\
				\hline
				8l+4 &4l+2\\
				\hline
				8l+5 &4l+3\\ 
				\hline
				8l+6 &4l+3\\ 
				\hline
				8l+7 &4l+3\\ 
				\hline
				8l+8 &4l+3\\ 
				\hline
			\end{tabular}
		\end{center}
	\end{thm}
	\begin{proof}
	See \cite[Theorem 3.1.1]{SV} and \cite[Proposition 5.6 and Theorem 8.3]{T2}.
	\end{proof}

	\section{Regular sequences in $H(BSO_n)$}
	
   In this section, we want to use the techniques developed in \cite[Section 7]{T2} to produce other regular sequences in the motivic cohomology of $BSO_n$ that will be relevant later to deal with the case of even Clifford groups.
   
    Let $V$ be an $n$-dimensional $\Z/2$-vector space, $B$ a bilinear form over $V$ and ${^{\perp}V}$ its right radical, i.e.
	$${^{\perp}V}=\{y \in V: B(x,y)=0 \: \: for \: \: any \: \: x \in V\}.$$
	
	Fix a basis $\{e_1,\dots,e_n\}$ for $V$ and let $x_i$ and $y_j$ be the coordinates of $x$ and $y$ in $V$, respectively. Then, $B(x,y)=\sum_{i,j=1}^n B(e_i,e_j)x_i y_j$ is a homogeneous polynomial of degree $2$ in $\Z/2[x_1,\dots,x_n,y_1,\dots,y_n]$. 
	
\begin{prop}\label{reg}
		The sequence $B(x,y), B(x,y^2), \dots, B(x,y^{2^{h-1}})$ is a regular sequence in the polynomial ring $\Z/2[x_1,\dots,x_n,y_1,\dots,y_n]$, where $h=n-dim({^{\perp}V})$.
	\end{prop}
	\begin{proof}
		See \cite[Corollary 7.3]{T2}.
	\end{proof}

Recall from \cite[Section 7]{T2} that there are commutative squares

	$$
\xymatrix{
	H(BO_{2m}) \ar@{->}[r]^(0.47){\alpha_{2m}} \ar@{->}[d]_{\gamma_{2m}} & H(BO_2)^{\otimes m} \ar@{->}[d]^{\delta_{2m}}\\
	S_{2m} \ar@{->}[r]_{\beta_{2m}} & R_{2m}
}
\hspace{1,5cm}
\xymatrix{
	H(BO_{2m+1}) \ar@{->}[r]^(0.38){\alpha_{2m+1}} \ar@{->}[d]_{\gamma_{2m+1}} & H(BO_2)^{\otimes m} \otimes H(BO_1) \ar@{->}[d]^{\delta_{2m+1}}\\
	S_{2m+1} \ar@{->}[r]_{\beta_{2m+1}} & R_{2m+1}
}
$$

where 
$$H(BO_n) \cong H[u_1, \dots,u_n],$$
$$H(BO_2)^{\otimes m} \cong H[x_1,y_1,\dots,x_m,y_m],$$
$$ H(BO_2)^{\otimes m} \otimes_H H(BO_1) \cong H[x_1,y_1,\dots,x_m,y_m,x_{m+1}],$$
$$S_n=\Z/2[u_1,\dots,u_n],$$
$$R_{2m}=\Z/2[x_1,y_1,\dots,x_m,y_m],$$ $$R_{2m+1}=\Z/2[x_1,y_1,\dots,x_m,y_m,x_{m+1}],$$
	$x_i$ is in bidegree $(0)[1]$ and $y_i$ is in bidegree $(1)[2]$ for any $i$, $\beta_n$ is obtained from $\alpha_n$ by tensoring with $\Z/2$ over $H$, $\gamma_n$ and $\delta_n$ are the reduction homomorphisms along $H \rightarrow \Z/2$. 
	
	In particular the following formulas hold:
	$$\beta_{2m}(u_{2j})=\sigma_j(y_1,\dots,y_m),$$
	$$\beta_{2m}(u_{2j+1})=\sum_{i=1}^m x_i\sigma_j(y_1,\dots,y_{i-1},y_{i+1},\dots,y_m),$$ 
	$$\beta_{2m+1}(u_{2j})=\sigma_j(y_1,\dots,y_m),$$ 
	$$\beta_{2m+1}(u_{2j+1})=\sum_{i=1}^m x_i\sigma_j(y_1,\dots,y_{i-1},y_{i+1},\dots,y_m)+x_{m+1}\sigma_j(y_1,\dots,y_m),$$
		where $\sigma_j$ is the $j$th elementary symmetric polynomial.
	
	\begin{lem}\label{tech}
		Let $f:A=\Z/2[a_1,\dots,a_m] \rightarrow B=\Z/2[b_1,\dots,b_n]$ be a ring homomorphism, where $deg(b_i)=1$ for any $i$ and $f(a_j)$ is a homogeneous polynomial in $B$ of positive degree $\alpha_j$ for any $j$. Moreover, let $r_1,\dots,r_k$ be a sequence of elements of $A$. If $f(r_1),\dots,f(r_k)$ is a regular sequence in $B$, then $r_1,\dots,r_k$ is a regular sequence in $A$.
	\end{lem}
	\begin{proof}
	See \cite[Lemma 7.4]{T2}.
	\end{proof}
	
	\begin{thm}\label{seq}
		The sequence $\gamma_n(u_1),\gamma_n(u_3),\gamma_n(Sq^2u_3),\dots,\gamma_n(Sq^{2^{l(n)-2}} Sq^{2^{l(n)-3}}\cdots Sq^2u_3)$ is regular in $S_n$, where $l(n)=[\frac{n+1} {2}]$. 
	\end{thm}
	\begin{proof}
		By Lemma \ref{tech} it is enough to show the regularity of the sequence
		$$\beta_n\gamma_n(u_1),\beta_n\gamma_n(u_3),\beta_n\gamma_n(Sq^2u_3),\dots,\beta_n\gamma_n(Sq^{2^{l(n)-2}} Sq^{2^{l(n)-3}}\cdots Sq^2u_3)$$
		in $R_n$.
		
		First, consider the case $n=2m$. Then, $\beta_{2m}\gamma_{2m}(u_1)=\beta_{2m}(u_1)=\sum_{i=1}^m x_i$. Moreover, since $\tau$ is zero in $R_{2m}$, we have that $\beta_{2m}\gamma_{2m}(u_3)=\sum_{i\neq j=1}^m x_iy_j$ and 
		\begin{align*}
			\beta_{2m}\gamma_{2m}(Sq^{2^l}\cdots Sq^2u_3)&=\delta_{2m}\alpha_{2m}(Sq^{2^l}\cdots Sq^2u_3)=\delta_{2m}(Sq^{2^l}\cdots Sq^2\alpha_{2m}(u_3))\\
			&=\sum_{i \neq j=1}^m \delta_{2m}(Sq^{2^l}\cdots Sq^2(x_iy_j))=\sum_{i\neq j=1}^m x_iy_j^{2^l}
		\end{align*}
		for $l \geq 1$. Modulo $\beta_{2m}\gamma_{2m}(u_1)$, $\beta_{2m}\gamma_{2m}(u_3)=B(x,y)=\sum_{i=1}^{m-1} x_i(y_i+y_m)$ is a bilinear form over an $m$-dimensional $\Z/2$-vector space $V$ and $\beta_{2m}\gamma_{2m}(Sq^{2^l}\cdots Sq^2u_3)=B(x,y^{2^l})$ for any $l \geq 1$.
		
		From $y_i+y_m=B(e_i,y)$ for any $i \leq m-1$, it follows that ${^{\perp}V} \cong \langle(1,\dots,1)\rangle$ and Proposition \ref{reg} implies that the sequence
		$$\beta_{2m}\gamma_{2m}(u_1),\beta_{2m}\gamma_{2m}(u_3),\beta_{2m}\gamma_{2m}(Sq^2u_3),\dots,\beta_{2m}\gamma_{2m}(Sq^{2^{l(2m)-2}} Sq^{2^{l(2m)-3}}\cdots Sq^2u_3)$$ 
		is regular in $R_{2m}$ where $l(2m)=m=[\tfrac{2m+1}{2}]$.
		
		Now, consider the case $n=2m+1$. Then, $\beta_{2m+1}\gamma_{2m+1}(u_1)=\beta_{2m+1}(u_1)=\sum_{i=1}^{m+1} x_i$. Moreover, since $\tau$ is zero in $R_{2m+1}$, we have that $\beta_{2m+1}\gamma_{2m+1}(u_3)=\sum_{i\neq j=1}^m x_iy_j+x_{m+1}\sum_{j=1}^m y_j$ and 
		\begin{align*}
			\beta_{2m+1}\gamma_{2m+1}(Sq^{2^l}\cdots Sq^2u_3)&=\delta_{2m+1}\alpha_{2m+1}(Sq^{2^l}\cdots Sq^2u_3)=\delta_{2m+1}(Sq^{2^l}\cdots Sq^2\alpha_{2m+1}(u_3))\\
			&=\sum_{i\neq j=1}^m \delta_{2m+1}(Sq^{2^l}\cdots Sq^2(x_iy_j))+\sum_{j=1}^m \delta_{2m+1}(Sq^{2^l}\cdots Sq^2(x_{m+1}y_j))\\
			&=\sum_{i \neq j =1}^m x_iy_j^{2^l} +\sum_{j =1}^m x_{m+1}y_j^{2^l}
		\end{align*}
			for $l \geq 1$. Modulo $\beta_{2m+1}\gamma_{2m+1}(u_1)$, $\beta_{2m+1}\gamma_{2m+1}(u_3)=B(x,y)=\sum_{i=1}^{m} x_iy_i$ is a bilinear form over an $m$-dimensional $\Z/2$-vector space $V$ and $\beta_{2m+1}\gamma_{2m+1}(Sq^{2^l}\cdots Sq^1u_2)=B(x,y^{2^l})$ for any $l \geq 1$. In this case ${^{\perp}V} \cong 0$, since $y_i=B(e_i,y)$ for any $i \leq m$, and Proposition \ref{reg} implies that the sequence 
		$$\beta_{2m+1}\gamma_{2m+1}(u_1),\beta_{2m+1}\gamma_{2m+1}(u_3),\beta_{2m+1}\gamma_{2m+1}(Sq^2u_3),\dots,\beta_{2m+1}\gamma_{2m+1}(Sq^{2^{l(2m+1)-2}} Sq^{2^{l(2m+1)-3}}\cdots Sq^2u_3)$$ 
		is regular in $R_{2m+1}$ where $l(2m+1)=m+1=[\frac{2m+2}{2}]$. This completes the proof. 
	\end{proof}

	\begin{cor}\label{tauseq}
		The sequence $\tau,\theta_1,\dots,\theta_{l(n)-1}$ is regular in $H(BSO_n)$, where $l(n)=[\frac{n+1} {2}]$.
	\end{cor}
	\begin{proof}
			Since $\theta_j$ is inductively computed from $\theta_1=u_3$ by using only Wu formula (see \cite[Proposition 5.7]{T2}) and Cartan formula, we know that $\theta_j$ is an element of $\Z/2[\tau,u_2,\dots,u_n]$ for any $j$. The regularity of the sequence in $\Z/2[\tau,u_2,\dots,u_n]$ follows from Theorem \ref{seq} by noticing that, modulo $\tau$ and $u_1$, $\theta_j=\gamma_n(Sq^{2^{j-1}}\cdots Sq^1u_2)$ in $S_n$. This clearly implies also the regularity of the sequence in $H(BSO_n)$.
	\end{proof}
	
	Recall from \cite[Section 7]{T2} the homomorphisms $i:H_{top}(BSO_n) \rightarrow H(BSO_n)$, $h:H_{top}(BSO_n) \rightarrow H(BSO_n)$ and $t:H(BSO_n) \rightarrow H_{top}(BSO_n)$, where $i$ is the ring homomorphism defined by $i(w_i)=u_i$, $h$ is the linear map defined by $h(x)=\tau^{[{\frac{p_{i(x)}} {2}}-q_{i(x)}]}i(x)$ for any monomial $x$, where $(q_{i(x)})[p_{i(x)}]$ is the bidegree of $i(x)$, and $t$ is the ring homomorphism defined by $t(u_i)=w_i$, $t(\tau)=1$ and $t(K_r^M(k)/2)=0$ for any $r>0$.
	
		\begin{lem}\label{h}
		For any homogeneous polynomials $x$ and $y$ in $H_{top}(BSO_n)$, we have that $h(xy)=\tau^{\epsilon}h(x)h(y)$, where $\epsilon$ is $1$ if $p_{i(x)}p_{i(y)}$ is odd and $0$ otherwise.
	\end{lem}
\begin{proof}
	See \cite[Lemma 7.7]{T2}.
	\end{proof}

\begin{lem}\label{th}
	For any $j$, $t(\theta_j)=\rho_j$ and $h(\rho_j)=\theta_j$.
\end{lem}
\begin{proof}
		See \cite[Lemma 7.9]{T2}.
\end{proof}

\begin{dfn}
	Let $I^{\circ}_j$ be the ideal in $H(BSO_n)$ generated by $\theta_1, \dots ,\theta_{j-1}$ and $I^{\circ,top}_j$ the ideal in $H_{top}(BSO_n)$ generated by $\rho_1, \dots ,\rho_{j-1}$.
	\end{dfn}

\begin{thm}\label{Q2}
	The canonical homomorphism 
	$$H_{top}(BSO_n)/I^{\circ,top}_{l(n)} \otimes \Z/2[e(\Delta_n)] \rightarrow H_{top}(BSpin^c_n)$$ 
	is an isomorphism, where $l(n)=[\frac{n+1} {2}]$ and $e(\Delta_n)$ is the Euler class of the complex spin representation $\Delta_n$.
\end{thm} 
\begin{proof}
	See \cite[Theorem 3.5]{HK}.
\end{proof}
	
The following is the main result of this section.
	
	\begin{thm}\label{MQ1}
		The sequence $\theta_1,\dots,\theta_{l(n)-1}$ is regular in $H(BSO_n)$ and $\theta_{l(n)} \in I^{\circ}_{l(n)}$, where $l(n)=[\frac{n+1} {2}]$.
	\end{thm}
	\begin{proof}
	Corollary \ref{tauseq} immediately implies that $\theta_1,\dots,\theta_{l(n)-1}$ is a regular sequence in $H(BSO_n)$.
	
	By Theorem \ref{Q2}, we know that $\rho_{l(n)}=Sq^{2^{l(n)-1}}\rho_{l(n)-1}$ vanishes in $H_{top}(BSpin^c_n)$, and so $\rho_{l(n)} \in I^{\circ,top}_{l(n)}$. It follows that $\rho_{l(n)}=\sum_{i=1}^{l(n)-1} \phi_i\rho_i$ for some homogeneous $\phi_i \in H_{top}(BSO_n)$ and, after applying $h$, we obtain that $\theta_{l(n)}=\sum_{i=1}^{l(n)-1} h(\phi_i)\theta_i$ by Lemmas \ref{h} and \ref{th}. Thus, $\theta_{l(n)} \in I^{\circ}_{l(n)}$, which completes the proof. 
	\end{proof}

\begin{rem}\label{kcasi}
	\normalfont
	Note that either $l(n)=k(n)$ or $l(n)=k(n)+1$. If $l(n)=k(n)$, then $\theta_{k(n)} \in I^{\circ}_{k(n)}$. On the other hand, if $l(n)=k(n)+1$, then the sequence $\theta_1,\dots,\theta_{k(n)}$ is regular in $H(BSO_n)$, and so $\theta_{k(n)} \notin I^{\circ}_{k(n)}$.
\end{rem}
	
	\section{The motivic cohomology ring of $B\Gamma^+_n$}

 	 In this last section, we prove our main result that describes the structure of the motivic cohomology of the Nisnevich classifying space of even Clifford groups.
 	 
 	 Before proceeding, recall from \cite[Section 3]{CM} that $\Gamma^+_n$-torsors are in one-to-one correspondence with quadratic forms with trivial discriminant and Clifford invariant, i.e. quadratic forms in $I^3$, where $I$ is the fundamental ideal of the Witt ring. Moreover, for any $n \geq 2$, we have the following short exact sequences of algebraic groups (see \cite[Chapter VI, Section 23.A]{KMRT})
 	 \begin{equation}\label{sesq}
	1 \rightarrow  \Gm \rightarrow \Gamma^+_n \rightarrow SO_n \rightarrow 1, \hspace{35pt} 1 \rightarrow Spin_n \rightarrow \Gamma^+_n \rightarrow \Gm \rightarrow 1.
\end{equation}
 	
 	\begin{lem}\label{ttben}
 		For any $n \geq 2$, $BSpin_n \cong \widehat{B}Spin_n$ with respect to the embedding $Spin_n \hookrightarrow \Gamma^+_n$.
 	\end{lem}
 \begin{proof}
 	First, note that, by \cite{FP} and \cite{P}, rationally trivial $Spin_n$-torsors and $\Gamma^+_n$-torsors are locally trivial. Moreover, recall from \cite[Section 4.1]{MV} that $Hom_{{\mathcal H}_s(k)}(Spec(K),B_{\acute{e}t}G) \cong H^1_{\acute et}(K,G)$ for any Nisnevich sheaf of groups $G$. Therefore, it follows from \cite[Section 3]{CM} that
 	$$Hom_{{\mathcal H}_s(k)}(Spec(K),B_{\acute{e}t}Spin_n) \rightarrow Hom_{{\mathcal H}_s(k)}(Spec(K),B_{\acute{e}t}\Gamma^+_n)$$
 	is surjective with trivial kernel, for any finitely generated field extension $K$ of $k$. Hence, we can apply Proposition \ref{BG} to the case that $G$ and $H$ are respectively $\Gamma^+_n$ and $Spin_n$, which provides the aimed result. 
 	\end{proof}
 
\begin{prop}\label{spincliff}
	For any $n \geq 2$, there exists a Gysin long exact sequence of $H(B\Gamma^+_n)$-modules
	$$\dots \rightarrow H^{*-1,*'}(BSpin_n) \xrightarrow{h^*} H^{*-2,*'-1}(B\Gamma^+_n) \xrightarrow{\text{$\cdot u_2$}} H^{*,*'}(B\Gamma^+_n) \xrightarrow{g^*} H^{*,*'}(BSpin_n) \rightarrow \dots$$
	such that the homomorphism $H^{2,*'}(BSO_n) \rightarrow H^{2,*'}(B\Gamma^+_n)$, induced by the map $\Gamma^+_n \rightarrow SO_n$ in (\ref{sesq}), is injective.
\end{prop}
\begin{proof}
	Let $M$ be $M(B\Gamma^+_n \rightarrow BSO_n)$ and $N$ be $Cone(M \rightarrow T)[-1]$ in ${\mathcal {DM}}^{-}_{eff}(BSO_n)$. From the motivic Serre spectral sequence (see \cite[Theorem 5.12]{T3}) associated to the sequence
	\begin{equation}\label{serre}
B\Gm \rightarrow B\Gamma^+_n \rightarrow BSO_n
	\end{equation}
	it follows that $H^{1,*'}(N) \cong 0$. Therefore, the homomorphism $H^{2,*'}(BSO_n) \rightarrow H^{2,*'}(B\Gamma^+_n) $ is injective. In particular, $u_2$ is non-trivial in $H^{2,1}(B\Gamma^+_n)$. 
	
	Now, consider the sequence
	$$\Gm \rightarrow BSpin_n \rightarrow B\Gamma^+_n.$$
	By Proposition \ref{Thom1}, Remark \ref{menumal} and Lemma \ref{ttben}, it induces a Gysin long exact sequence of $H(B\Gamma^+_n)$-modules
	$$\dots \rightarrow H^{p-1,q}(BSpin_n) \xrightarrow{h^*} H^{p-2,q-1}(B\Gamma^+_n) \xrightarrow{f^*} H^{p,q}(B\Gamma^+_n) \xrightarrow{g^*} H^{p,q}(BSpin_n) \xrightarrow{} \dots.$$
    Since $g^*$ is an isomorphism in bidegree $(1)[1]$, we have that $f^*(1)$ is the only non-trivial class in $H^{2,1}(B\Gamma^+_n)$ that vanishes in $H^{2,1}(BSpin_n)$. It follows that $f^*(1)=u_2$, which completes the proof.
	\end{proof}

\begin{lem}\label{surjme}
	For any $n \geq 2$, $H^{p,*'}(BSO_n) \rightarrow H^{p,*'}(B\Gamma^+_n)$ is surjective for $p < 2^{k(n)}$.
\end{lem}
\begin{proof}
	We proceed by induction on $p$. For $p=0$, the Serre spectral sequence associated to \ref{serre} implies that $H^{0,*'}(BSO_n) \cong H^{0,*'}(B\Gamma^+_n)$, which provides the induction basis. Now, suppose that $H(BSO_n) \rightarrow H(B\Gamma^+_n)$ is surjective in topological degrees less than $p < 2^{k(n)}$, and consider a class $x$ in $H^{p,*'}(B\Gamma^+_n)$. Since by Theorem \ref{bso} the homomorphism $H^{p,*'}(BSO_n) \rightarrow H^{p,*'}(BSpin_n)$ that factors through $H^{p,*'}(B\Gamma^+_n)$ is surjective for $p < 2^{k(n)}$, we have that $g^*(x)=g^*(y)$ for some $y$ in the image of $H^{p,*'}(BSO_n) \rightarrow H^{p,*'}(B\Gamma^+_n)$. Hence, by Proposition \ref{spincliff}, there exists a class $z$ in $H^{p-2,*'-1}(B\Gamma^+_n)$ such that $x=y+u_2z$. By induction hypothesis, $z$ is in the image of $H^{p-2,*'-1}(BSO_n) \rightarrow H^{p-2,*'-1}(B\Gamma^+_n)$, from which it follows that $x$ is in the image of $H^{p,*'}(BSO_n) \rightarrow H^{p,*'}(B\Gamma^+_n)$ that is what we wanted to show.
	\end{proof}

\begin{dfn}
	Denote by $\omega_n$ the class $h^*(v_{2^{k(n)}})$ in $H^{2^{k(n)}-1,2^{k(n)-1}-1}(B\Gamma^+_n)$.
\end{dfn}

\begin{rem}\label{doppiovu}
	\normalfont
	  It follows from Lemma \ref{surjme} that $\omega_n$ belongs to the image of $H(BSO_n) \rightarrow H(B\Gamma^+_n)$. Moreover, Proposition \ref{spincliff} implies that $u_2\omega_n=0$ in $H(B\Gamma^+_n)$.
\end{rem}

\begin{prop}\label{n2}
The motivic cohomology ring of $B\Gamma^+_2$ is given by
$$H(B\Gamma^+_2) \cong H[u_2,e_2],$$
where $e_2$ is a lift of $v_2$ in $H(BSpin_2)$ under the homomorphism $H(B\Gamma^+_2) \rightarrow H(BSpin_2)$.
\end{prop}
\begin{proof}
	Consider the Gysin long exact sequence from Proposition \ref{spincliff} 
	$$\dots \rightarrow H^{*-1,*'}(BSpin_2) \xrightarrow{h^*} H^{*-2,*'-1}(B\Gamma^+_2) \xrightarrow{\text{$\cdot u_2$}} H^{*,*'}(B\Gamma^+_2) \xrightarrow{g^*} H^{*,*'}(BSpin_2) \rightarrow \dots.$$
	Since $H(BSpin_2) \cong H[v_2]$, with $v_2$ in bidegree $(1)[2]$, $g^*$ is a ring homomorphism and $H^{1,0}(B\Gamma^+_2) \cong 0$, we have that $h^*$ is zero, the multiplication by $u_2$ is injective in $H(B\Gamma^+_2)$ and the quotient of $H(B\Gamma^+_2)$ modulo the ideal generated by $u_2$ is $H[v_2]$. This concludes the proof.
	\end{proof}
	
	\begin{lem}\label{u3}
		For any $n \geq 3$, $u_3=0$ in $H(B\Gamma^+_n)$. Moreover, there exists a unique element $x_1$ in $H^{2,1}(N)$ that maps to $u_3$ in $H^{3,1}(BSO_n)$.
	\end{lem}
	\begin{proof}
By Proposition \ref{spincliff} we have a Gysin long exact sequence
$$\dots \rightarrow H^{*-1,*'}(BSpin_n) \xrightarrow{h^*} H^{*-2,*'-1}(B\Gamma^+_n) \xrightarrow{\text{$\cdot u_2$}} H^{*,*'}(B\Gamma^+_n) \xrightarrow{g^*} H^{*,*'}(BSpin_n) \rightarrow \dots.$$
Since $u_3$ is trivial in $H(BSpin_n)$ and $H^{1,0}(B\Gamma^+_n) \cong 0$, we have that $u_3$ is trivial also in $H(B\Gamma^+_n)$. Moreover, note that, for $n \geq 3$, $H^{2,1}(B\Gamma^+_n) \cong \Z/2 \cdot u_2$.

Now, consider the long exact sequence
$$\dots \rightarrow H^{*-1,*'}(B\Gamma^+_n) \rightarrow H^{*-1,*'}(N) \rightarrow H^{*,*'}(BSO_n) \rightarrow H^{*,*'}(B\Gamma^+_n) \rightarrow \dots .$$
The homomorphism $H^{2,1}(BSO_n) \rightarrow H^{2,1}(B\Gamma^+_n)$ is bijective. Since $u_3$ is trivial in $H(B\Gamma^+_n)$, we deduce that $H^{2,1}(N) \rightarrow H^{3,1}(BSO_n)$ is an isomorphism, which finishes the proof.
	\end{proof}

	\begin{dfn}
		For any $j \geq 2$ and $n \geq 3$, let $x_j$ be the class in $H^{2^j,2^{j-1}}(N)$ defined by $x_j=Sq^{2^{j-1}} \cdots Sq^2x_1$ and denote by $\langle x_1, \dots ,x_{j-1}\rangle  $ the $H(BSO_n)$-submodule of $H(N)$ generated by $x_1, \dots ,x_{j-1}$. 
		
	\end{dfn}
	
	\begin{lem}\label{x}
		For any $j \geq 2$ and $n \geq 3$, $x_j \notin \langle x_1, \dots ,x_{j-1}\rangle$.
	\end{lem}
	\begin{proof}
		This follows by noticing that $x_j$ maps to the respective class defined for spin groups in \cite[Lemma 8.2]{T2}.
	\end{proof}

\begin{prop}\label{vp}
	Suppose there exists a class $e$ in $H(B\Gamma^+_n)$ such that $g^*(e)$ is a monic homogeneous polynomial $c$ in $v_{2^{k(n)}}$ with coefficients in $H(BSO_n)$, and denote by $p$ the obvious homomorphism $H(BSO_n) \otimes_H H[e] \rightarrow H(B\Gamma^+_n)$.
	
1)	If $\Ima(h^*)=\Ima(p) \cdot \omega_n$, then $\ker(p)=J^{\circ}_{k(n)}+(u_2 \omega_n)$, where $J^{\circ}_{k(n)}$ is $I^{\circ}_{k(n)} \otimes_H H[e]$.
	
2)	If moreover $\ker(h^*)=\Ima(g^*p)$, then there is an isomorphism
	$$H(BSO_n)/(I^{\circ}_{k(n)}+(u_2 \omega_n)) \otimes_H H[e] \rightarrow H(B\Gamma^+_n).$$
\end{prop}
\begin{proof}
We start by proving 1). It immediately follows from Remark \ref{doppiovu} and Lemma \ref{u3} that $J^{\circ}_{k(n)}+(u_2 \omega_n) \subseteq \ker(p)$. We show the opposite inclusion by induction on the topological degree. Proposition \ref{spincliff} provides the induction basis. Now, suppose that $x$ is in $\ker(p)$ and every class in $\ker(p)$ with topological degree less than the topological degree of $x$ belongs to $J^{\circ}_{k(n)}+(u_2 \omega_n)$. We can write $x$ as $\sum_{j=0}^m \phi_je^j$ for some $\phi_j \in H(BSO_n)$. Then, $\sum_{j=0}^m \phi_j c^j=g^*p(x)=0$, and so $\phi_j=0$ in $H(BSpin_n)$ for any $j$ since by hypothesis $c$ is a monic polynomial in $v_{2^{k(n)}}$. Therefore, $\phi_j \in I_{k(n)}= I^{\circ}_{k(n)}+(u_2)$ by Theorem \ref{bso}. Hence, there are $\psi_j \in H(BSO_n)$ such that $\phi_j+u_2\psi_j \in I^{\circ}_{k(n)}$, from which it follows that $x+u_2z \in J^{\circ}_{k(n)}$ where $z=\sum_{j=0}^m \psi_je^j$. Thus, $u_2p(z)=0$ which implies that $p(z) \in \Ima(h^*)=\Ima(p) \cdot \omega_n$, and so there exists an element $y$ in $H(BSO_n) \otimes_H H[e]$ such that $p(z)=p(y\omega_n)$. Therefore, $z+y\omega_n \in J^{\circ}_{k(n)}+(u_2\omega_n)$ by induction hypothesis. It follows that $z \in J^{\circ}_{k(n)}+(\omega_n)$ and $x \in J^{\circ}_{k(n)}+(u_2\omega_n)$.

We now move to $2)$. We prove  by induction on the topological degree that, if $\ker(h^*)=\Ima(g^*p)$, then $\Ima(p)=H(B\Gamma^+_n)$. Lemma \ref{surjme} provides the induction basis. Let $x$ be a class in $H(B\Gamma^+_n)$ and suppose that $p$ is an epimorphism in topological degrees less than the topological degree of $x$. From $g^*(x) \in \ker(h^*)=\Ima(g^*p)$ it follows that there is an element $\chi$ in $H(BSO_n) \otimes_H H[e]$ such that $g^*(x)=g^*p(\chi)$. Therefore, $x+p(\chi)=u_2z$ for some $z \in H(B\Gamma^+_n)$. By induction hypothesis $z=p(\zeta)$ for some element $\zeta \in H(BSO_n) \otimes_H H[e]$, hence $x=p(\chi+u_2\zeta)$ that is what we aimed to show. 
\end{proof}

	\begin{rem}\label{rem2}
	\normalfont
	Since $H(BSpin_n)$ is generated by the powers $v_{2^{k(n)}}^i$ as a $H(BSO_n)$-module, we have that $\Ima(h^*)$ is generated by $h^*(v_{2^{k(n)}}^i)$ as a $H(BSO_n)$-module.
\end{rem} 

\begin{lem}\label{Sqw}
	For any $m \geq 0$, we have $Sq^m\omega_n \in \langle \omega_n \rangle  $, where $\langle \omega_n\rangle  $ is the $H(BSO_n)$-submodule of $H(B\Gamma^+_n)$ generated by $\omega_n$.
\end{lem}
\begin{proof}
	We proceed by induction on $m$. For $m=0$ we have that $Sq^0\omega_n=\omega_n$ and for $m>2^{k(n)}-1$ we have that $Sq^m\omega_n=0$ by \cite[Corollary 5.8]{T2}. Suppose that $Sq^i\omega_n \in \langle \omega_n \rangle$ for $i < m \leq 2^{k(n)}-1$. Then, by Cartan formula, in $H(BSO_n)$ we have that
	$$Sq^m(u_2\omega_n)=u_2Sq^m\omega_n+\tau u_3Sq^{m-1}\omega_n+u_2^2Sq^{m-2}\omega_n.$$
	Therefore, from Remark \ref{doppiovu} and Lemma \ref{u3} we deduce that $u_2Sq^m\omega_n=0$ 
	in $H(B\Gamma^+_n)$, since by induction hypothesis  $Sq^{m-2}\omega_n \in \langle \omega_n \rangle$. It follows that $Sq^m\omega_n\in \Ima(h^*)$. By Remark \ref{rem2}, we know that $Sq^m\omega_n=\sum_{i \geq 1} \phi_i h^*(v_{2^{k(n)}}^i)$ for some $\phi_i \in H(BSO_n)$. But, for any $i\geq 2$, the topological degree of $h^*(v_{2^{k(n)}}^i)$ is $i2^{k(n)}-1 >2^{k(n)+1}-2 \geq m + 2^{k(n)} -1$ that is the topological degree of $Sq^m\omega_n$. Hence, $Sq^m\omega_n=\phi_1 h^*(v_{2^{k(n)}}) =\phi_1 \omega_n $ that is what we aimed to prove. 
\end{proof}

\begin{lem}\label{lambdamu}
	It exists an element $\mu_n$ in $H(B\Gamma^+_n)$ such that $v_{2^{k(n)}}^j=g^*(\mu_n^{[\frac j 2]})v_{2^{k(n)}}^{j-2[\frac j 2]}$ in $H(BSpin_n)$ for any $j \geq 0$.
\end{lem}
\begin{proof}
	For $j=0,1$ the statement is tautological. For $j=2$, by Cartan formula, we have that
	$$h^*(v_{2^{k(n)}}^2)=h^*(Sq^{2^{k(n)}}v_{2^{k(n)}})=Sq^{2^{k(n)}}(\omega_n\alpha)=\sum_{i=0}^{2^{k(n)}}\tau^{i \: mod 2}Sq^{2^{k(n)}-i}\omega_n Sq^i\alpha$$
	where $\alpha$ is the Thom class of the map $BSpin_n \rightarrow B\Gamma^+_n$. Note that, by Proposition \ref{spincliff} and Lemma \ref{surjme}, $H^{p-2,*'-1}(B\Gamma^+_n) \xrightarrow{\cdot u_2} H^{p,*'}(B\Gamma^+_n)$ is a monomorphism for $p \leq 2^{k(n)}$. This implies, in particular, that $Sq^1\alpha=0$ and $Sq^2\alpha=u_2\alpha$. Moreover, $Sq^i\alpha=0$ for $i \geq 3$, since $\alpha$ is in bidegree $(1)[2]$, and $Sq^{2^{k(n)}}\omega_n=0$. Therefore, we have that
	$$h^*(v_{2^{k(n)}}^2)=Sq^{2^{k(n)}-2}\omega_nu_2\alpha=0$$
	since, by Lemma \ref{Sqw}, $Sq^{2^{k(n)}-2}\omega_n\in \langle \omega_n\rangle $ and $u_2\omega_n=0$ by Remark \ref{doppiovu}. Hence, $v_{2^{k(n)}}^2 \in \Ima(g^*)$.
	
	Let $\mu_n$ be a class in $H(B\Gamma^+_n)$ such that $g^*(\mu_n)=v_{2^{k(n)}}^2$. Suppose the statement is true for $i < j$, then
	$$v_{2^{k(n)}}^j=v_{2^{k(n)}}^2v_{2^{k(n)}}^{j-2}=g^*(\mu_n)g^*(\mu_n^{[\frac {j-2} 2]})v_{2^{k(n)}}^{j-2-2[\frac {j-2} 2]}=g^*(\mu_n^{[\frac j 2]})v_{2^{k(n)}}^{j-2[\frac j 2]}$$
	that concludes the proof.
\end{proof}

\begin{rem}\label{rem3}
	\normalfont
	It immediately follows from Lemma \ref{lambdamu} that 
	\begin{align*}
		h^*(v_{2^{k(n)}}^j)=
		\begin{cases}
			0, &for \: j \: even;\\
			\mu_n^{\frac{j-1}2}\omega_n, & for \: j \: odd.
			\end{cases}
		\end{align*}
\end{rem}

The following is the main result of this paper.

	\begin{thm}\label{MQ2}
		For any $n \geq 2$, there exists a cohomology class $e_{2^{l(n)}}$ in bidegree $(2^{l(n)-1})[2^{l(n)}]$ such that the natural homomorphism of $H$-algebras 
		$$H(BSO_n)/I^{\circ}_{l(n)} \otimes_H H[e_{2^{l(n)}}] \rightarrow H(B\Gamma^+_n)$$
		is an isomorphism, where $I^{\circ}_{l(n)}$ is the ideal generated by $\theta_1, \dots ,\theta_{l(n)-1}$ and $l(n)=[\frac{n+1}{2}]$.
	\end{thm}
	\begin{proof}
	For $n=2$ this is given by Proposition \ref{n2}, so suppose from now on that $n \geq 3$.
	
	If $\omega_n=0$, then there is a class $e_{2^{k(n)}}$ in $H(B\Gamma^+_n)$ such that $g^*(e_{2^{k(n)}})=v_{2^{k(n)}}$. Let $p$ be the homomorphism $H(BSO_n) \otimes_H H[e_{2^{k(n)}}] \rightarrow H(B\Gamma^+_n)$. Then, $\Ima(h^*)=0=\Ima(p) \cdot \omega_n$ and $\ker(h^*)=H(BSpin_n)=\Ima(g^*p)$. Hence, Proposition \ref{vp} implies that the homomorphism
		$$H(BSO_n)/I^{\circ}_{k(n)} \otimes_H H[e_{2^{k(n)}}] \rightarrow H(B\Gamma^+_n)$$
		is an isomorphism. Since $\theta_{k(n)}$ vanishes in $H(B\Gamma^+_n)$ we have that $\theta_{k(n)} \in I^{\circ}_{k(n)}$, which means that $k(n)=l(n)$ by Remark \ref{kcasi}.
		
	If $\omega_n \neq 0$, then set $p:H(BSO_n) \otimes_H H[\mu_n] \rightarrow H(B\Gamma^+_n)$ where $\mu_n$ is the class from Lemma \ref{lambdamu}. It follows from Remarks \ref{rem2} and \ref{rem3} that $\Ima(h^*)=\Ima(p) \cdot \omega_n$. Then, by Proposition \ref{vp}, we obtain that $\ker(p)=J^{\circ}_{k(n)}+(u_2\omega_n)$.
		
		 Since $\omega_n\neq 0$, we can extend the result in Lemma \ref{surjme} to the degree $p=2^{k(n)}$, i.e. we have that $H^{2^{k(n)},2^{{k(n)}-1}}(BSO_n) \rightarrow H^{2^{k(n)},2^{{k(n)}-1}}(B\Gamma^+_n)$ is surjective. Hence, $H^{2^{k(n)},2^{{k(n)}-1}}(B\Gamma^+_n) \rightarrow H^{2^{k(n)},2^{{k(n)-1}}}(N)$ 
		is zero and $H^{2^{k(n)},2^{{k(n)-1}}}(N) \rightarrow H^{2^{k(n)}+1,2^{{k(n)}-1}}(BSO_n)$ is injective. It follows that $\theta_{k(n)} \notin I^{\circ}_{k(n)}$, since $x_{k(n)} \notin \langle x_1, \dots ,x_{{k(n)}-1}\rangle$ by Lemma \ref{x}, and $k(n)+1=l(n)$ by Remark \ref{kcasi}. Observe that, as we have already shown, $\ker(p)=J^{\circ}_{k(n)}+(u_2\omega_n)$ and $\theta_{k(n)}$ vanishes in $H(B\Gamma^+_n)$. Therefore, $\theta_{k(n)}+u_2\omega_n \in I^{\circ}_{k(n)}$, which implies that $\ker(p)=J^{\circ}_{{k(n)}+1}=J^{\circ}_{l(n)}$.
		
		Now, it remains to prove that $\ker(h^*)=\Ima(g^*p)$. Obviously, $\Ima(g^*p) \subseteq \ker(h^*)$, so we only have to prove the other side inclusion. Let $x$ be an element of $\ker(h^*)$. We can write $x$ as $\sum_{j=0}^m\gamma_jv_{2^{k(n)}}^j$ with $\gamma_j \in H(BSO_n)$. Then, by Remark \ref{rem3}, we have that $\sum_{j=1,odd}^m\gamma_j\mu_n^{{\frac {j-1}2 }}\omega_n=0$. Denote by $\sigma$ the element $\sum_{j=1,odd}^m\gamma_j\mu_n^{{\frac {j-1}2 }}$ in $H(BSO_n) \otimes_H H[\mu_n]$. From $p(\sigma \omega_n)=0$ we deduce that $\sigma \omega_n \in J^{\circ}_{{k(n)}+1}$, since we have shown that $\ker(p)=J^{\circ}_{{k(n)}+1}$. Thus, $\sigma \omega_n=\sum_{j=1}^{k(n)}\sigma_j\theta_j$ for some $\sigma_j \in H(BSO_n) \otimes_H H[\mu_n]$ and, multiplying by $u_2$, we obtain that $u_2\sigma \omega_n+u_2\sigma_{k(n)} \theta_{k(n)} \in J^{\circ}_{k(n)}$. On the other hand, $\theta_{k(n)}+u_2\omega_n \in I^{\circ}_{k(n)}$, from which it follows by multiplying by $\sigma$ that $\sigma \theta_{k(n)}+u_2\sigma \omega_n \in J^{\circ}_{k(n)}$. Hence, $(\sigma+u_2\sigma_{k(n)})\theta_{k(n)} \in J^{\circ}_{k(n)}$. Theorem \ref{MQ1} implies that $\sigma+u_2\sigma_{k(n)} \in J^{\circ}_{k(n)}$, from which it follows that $\sigma \in J^{\circ}_{k(n)}+(u_2)=J_{k(n)}$. Therefore, $g^*p(\sigma)=0$ in $H(BSpin_n)$ and by Lemma \ref{lambdamu}
			$$x=\sum_{j=1,odd}^m\gamma_jg^*(\mu_n^{\frac {j-1} 2})v_{2^{k(n)}} +\sum_{j=0,even}^m\gamma_jg^*(\mu_n^{\frac j 2})=g^*p(\sigma)v_{2^{k(n)}}+\sum_{j=0,even}^m\gamma_jg^*(\mu_n^{\frac j 2})=\sum_{j=0,even}^m\gamma_jg^*(\mu_n^{\frac j 2})$$
			is an element of $\Ima(g^*p)$.
	
	    Rename the class $\mu_n$ by $e_{2^{l(n)}}$. Then, by Proposition \ref{vp} we have that the homomorphism
		$$H(BSO_n)/I^{\circ}_{{l(n)}} \otimes_H H[e_{2^{{l(n)}}}] \rightarrow H(B\Gamma^+_n)$$
		is an isomorphism, and the proof is complete.
	\end{proof}

\begin{dfn}
Denote by ${\mathrm Chern}(B_{\acute{e}t}\Gamma^+_n)$ the subring of the Chow ring with $\Z/2$-coefficients $Ch(B_{\acute{e}t}\Gamma^+_n)$ generated by the Chern classes of the representation $\Gamma^+_n \rightarrow SO_n$.
\end{dfn} 

For any $2 \leq i \leq n$, let $\widetilde{w}_i$ be the Stiefel-Whitney class in $H^{i,i}(B_{\acute et} SO_n)$. Recall from \cite[Theorem 3.1.1]{SV} that the homomorphism $H(B_{\acute et} SO_n) \rightarrow H(BSO_n)$, induced by the canonical map $BSO_n \rightarrow B_{\acute et}SO_n$, sends $\widetilde{w}_i$ to $\tau^{[{\frac{i+1}{2}}]}u_i$.

\begin{lem}\label{w3}
	The homomorphism $H(B_{\acute et} SO_n) \rightarrow H(B_{\acute{e}t}\Gamma^+_n)$ maps $Sq^1\widetilde{w}_2$ to $0$.
	\end{lem}
\begin{proof}
	Note that the homomorphism $H^{3,2}(B_{\acute{e}t}\Gamma^+_n) \rightarrow H^{3,2}(B\Gamma^+_n)$ is injective, since the change of topology $H^{3,2}(B_{\acute{e}t}\Gamma^+_n) \rightarrow H_{\acute et}^{3,2}(B_{\acute et}\Gamma^+_n) \cong H_{\acute et}^{3,2}(B\Gamma^+_n)$, which factors through $H^{3,2}(B\Gamma^+_n)$, is a monomorphism by \cite[Corollary 6.9]{V3}.
	
	On the other hand, the homomorphism $H^{3,2}(B_{\acute{e}t}SO_n) \rightarrow H^{3,2}(BSO_n)$ maps $Sq^1\widetilde{w}_2$ to $Sq^1(\tau u_2)=\tau u_3$ that vanishes in $H^{3,2}(B\Gamma^+_n)$. Hence, $Sq^1\widetilde{w}_2$ maps to 0 in $H^{3,2}(B_{\acute{e}t}\Gamma^+_n)$ that completes the proof.
	\end{proof}

\begin{rem}\label{mserv}
	\normalfont
	As noted in \cite[Remark 11.3]{T2}, the class $\tau \theta_i^2$ belongs to the Chern subring  ${\mathrm Chern}(B_{\acute{e}t}SO_n) \cong \Z/2[c_2,\dots,c_n]$, for any $i \geq 1$.
	
	Then, \cite[Lemma 11.2]{T2} and Lemma \ref{w3} imply that $\tau \theta_i^2 = \tau Sq^1\theta_{i+1} =
	Sq^1Sq^{2^i} \cdots Sq^1\widetilde{w}_2$ vanishes in ${\mathrm Chern}(B_{\acute{e}t}\Gamma^+_n)$ for all $i \geq 1$.
\end{rem}

The following result provides a complete description of ${\mathrm Chern}(B_{\acute{e}t}\Gamma^+_n)$ modulo nilpotents.

\begin{cor}\label{chern}
	There exists a ring isomorphism
	$${\mathrm Chern}(B_{\acute{e}t}\Gamma^+_n)_{red}\cong \Z/2[c_2,\dots,c_n]/\sqrt{(\tau \theta_1^2,\dots,\tau \theta_{l(n)-1}^2)}$$
	where $c_i=\tau^{i \: mod2}u_i^2$ is the {\emph i}th Chern class in $H(BSO_n)$.
\end{cor}
\begin{proof}
		Let ${\mathcal I}^{\circ}_n$ be the kernel of the epimorphism $\Z/2[c_2,\dots,c_n] \rightarrow {\mathrm Chern}(B_{\acute{e}t}\Gamma^+_n)$. Then, by Remark \ref{mserv} we have that $(\tau \theta_1^2,\dots,\tau \theta_{l(n)-1}^2) \subseteq {\mathcal I}^{\circ}_n$. On the other hand, since the epimorphism $\Z/2[c_2,\dots,c_n] \rightarrow {\mathrm Chern}(B\Gamma^+_n)$ factors through ${\mathrm Chern}(B_{\acute{e}t}\Gamma^+_n)$, Theorem \ref{MQ2} implies that ${\mathcal I}^{\circ}_n \subseteq \iota^{-1}(I^{\circ}_{l(n)})$, where $\iota:\Z/2[c_2,\dots,c_n] \rightarrow H(BSO_n)$ is the inclusion of the Chern subring of $H(BSO_n)$. 
		
		Now, observe that $\sqrt{(\tau \theta_1^2,\dots,\tau \theta_{l(n)-1}^2)} = \sqrt{\iota^{-1}(I^{\circ}_{l(n)})}$. Therefore, we obtain that
		$${\mathrm Chern}(B_{\acute{e}t}\Gamma^+_n)_{red}\cong \Z/2[c_2,\dots,c_n]/\sqrt{{\mathcal I}^{\circ}_n}\cong \Z/2[c_2,\dots,c_n]/\sqrt{(\tau \theta_1^2,\dots,\tau \theta_{l(n)-1}^2)}$$
		that is what we aimed to show.
	\end{proof}

\begin{rem}
	\normalfont 
	Note that the relations appearing in Corollary \ref{chern} are also expressible in terms of the action of some Steenrod operations on $c_2$. More precisely, we have that $\tau\theta_j^2=Sq^{2^j}Sq^{2^{j-1}} \cdots Sq^4Sq^2c_2$, for any $j \geq 1$.
\end{rem}
	
\footnotesize{

\noindent {\scshape Fachbereich Mathematik, Technische Universit{\"a}t Darmstadt}\\
fabio.tanania@gmail.com
	

\begin{thebibliography}{00}
		
		\bibitem{CM} V. Chernousov, A. Merkurjev,
		\textit{Essential dimension of spinor and Clifford groups}, Algebra Number Theory 8 (2014), no. 2, 457-472.
		
		\bibitem{FP} R. Fedorov, I. Panin,
		\textit{A proof of the Grothendieck-Serre conjecture on principal bundles over regular local rings containing infinite fields},
		Publ. Math. Inst. Hautes \'Etudes Sci. 122 (2015), 169-193.
		
        \bibitem{HK} M. Harada, A. Kono, \textit{Cohomology mod 2 of the classifying space of $Spin^c(n)$}, Publ. Res. Inst. Math. Sci. 22 (1986), no. 3, 543-549.
        
        \bibitem{K} N. Karpenko,
        \textit{On special Clifford groups and their characteristic classes}, available at\\
       https://sites.ualberta.ca/\texttildelow karpenko/publ/spind2-16.pdf.

        \bibitem{KMRT} M.-A. Knus, A. Merkurjev, M. Rost, J.-P. Tignol,
		\textit{The book of involutions}, American Mathematical Society Colloquium Publications, 44. American Mathematical Society, Providence, RI, 1998.
		
		\bibitem{MV} F. Morel, V. Voevodsky,
		\textit{$A^1$-homotopy theory of schemes}, Publications Math$\mathrm{\acute{e}}$matiques de l'I.H.$\mathrm{\acute{E}}$.S. no. 90.
		
		\bibitem{P} I. Panin,
		\textit{Proof of Grothendieck-Serre conjecture on principal bundles over regular local rings containing a finite field}, arXiv:1707.01767.
		
		\bibitem{SV} A. Smirnov, A. Vishik,
		\textit{Subtle Characteristic Classes}, arXiv:1401.6661.
		
		\bibitem{T3} F. Tanania,
		\textit{A Serre-type spectral sequence for motivic cohomology}, arXiv:2208.03254.
		
		\bibitem{T1} F. Tanania,
		\textit{Subtle characteristic classes and hermitian forms}. Doc. Math. 24 (2019), 2493-2523.
		
		\bibitem{T2} F. Tanania,
		\textit{Subtle characteristic classes for $Spin$-torsors}, Math. Z. 301 (2022), no. 1, 41-74.
		
		\bibitem{To} B. Totaro,
		\textit{The Chow ring of a classifying space}, Algebraic K-theory (Seattle, WA, 1997), 249-281, Proc. Sympos. Pure Math., 67, Amer. Math. Soc., Providence, RI, 1999.
		
		\bibitem{V1} V. Voevodsky,
		\textit{Motives over simplicial schemes}, J. K-Theory 5 (2010), no. 1, 1-38.
		
		\bibitem{V3} V. Voevodsky,
		\textit{Motivic cohomology with ${\mathbf Z}/2$-coefficients}, Publ. Math. Inst. Hautes $\mathrm{\acute{E}}$tudes Sci., pp. 59-104, 2003.
		
		\bibitem{V2} V. Voevodsky,
		\textit{Reduced power operations in motivic cohomology}, Publ. Math. Inst. Hautes $\mathrm{\acute{E}}$tudes Sci., pp. 1-57, 2003.\\

		\end{thebibliography}
\end{document}